\documentclass[11pt,oneside,leqno]{amsart}
\usepackage{amsxtra}
\usepackage[all]{xy}
\usepackage{amsopn}
\usepackage{color}
\usepackage{amsmath,amsthm,amssymb}
\usepackage{amscd}
\usepackage{amsfonts}
\usepackage{latexsym}
\usepackage{verbatim}
\usepackage{pb-diagram}
\usepackage{tikz}

\newtheorem{theorem}{Theorem}[section]
\newtheorem{corollary}{Corollary}[section]
\newtheorem*{theorem*}{Theorem}
\newtheorem*{corollary*}{Corollary}
\newtheorem{definition}[theorem]{Definition}
\newtheorem{rem}[theorem]{Remark}

\newcommand{\del}{\partial}
\newcommand{\delbar}{\overline{\del}}

\begin{document}
\title[Relative \v{C}ech-Dolbeault homology and applications]
{Relative \v{C}ech-Dolbeault homology and applications}
\author{Nicoletta Tardini}
\date{\today}
\address{Dipartimento di Matematica e Informatica ``Ulisse Dini''\\
Universit\`a degli Studi di Firenze\\
viale Morgagni 67/a\\
50134 Firenze, Italy}
\email{nicoletta.tardini@gmail.com}
\thanks{Partially supported by SIR2014 project RBSI14DYEB ``Analytic aspects in complex and hypercomplex geometry'' and by GNSAGA
of INdAM}
\begin{abstract}
We define the relative Dolbeault homology of a complex manifold
with currents via a \v{C}ech approach
and we prove its equivalence with the relative \v{C}ech-Dolbeault cohomology
as defined in \cite{suwa}. This definition is then used to compare the relative
Dolbeault cohomology groups of two complex manifolds of the same dimension
related by a suitable proper surjective holomorphic map. Finally, an application
to blow-ups is considered and a blow-up formula for the Dolbeault cohomology in terms of relative cohomology is presented. 
\end{abstract}

\maketitle

\section*{Introduction}

On a complex manifold an important global invariant is represented by the
Dolbeault cohomology which can be described equivalently using complex
differential forms or currents. In particular, this double interpretation was
used fruitfully by R. O. Wells in \cite{wells} in order to compare the
Dolbeault cohomology of two complex manifolds
of the same dimensions related by a proper holomorphic surjective map.
More precisely, he proved that if $\pi:\tilde X\to X$ is a proper holomorphic surjective map between two complex manifolds of the same dimension then the induced map in cohomology
$$
\pi^*:H^{\bullet,\bullet}_{\delbar}(X)\to H^{\bullet,\bullet}_{\delbar}(\tilde X)
$$
is injective. In particular, if $X$ and $\tilde X$ are compact
we have the dimensional inequalities $h_{\delbar}^{\bullet,\bullet}(X)
\leq h_{\delbar}^{\bullet,\bullet}(\tilde X)$ for their Hodge numbers.
In fact, this result can be weaken to almost-complex manifolds as done in
\cite{tardini-tomassini-almost-complex} for pseudo-holomorphic maps, where the Dolbeault cohomology is replaced by the cohomology groups introduced by Li and Zhang \cite{li-zhang}.\\
The Dolbeault cohomology can also be interpreted via a \v{C}ech approach, as done in \cite{suwa}, and it turns out that this approach is also useful
in order to define its relative version. The relative \v{C}ech-Dolbeault cohomology,
which is equivalent to the relative Dolbeault cohomology as defined for instance in \cite{ida} (cf. \cite{suwa-sheaf}),
can be used to describe the localization theory of characteristic classes
(\cite{suwa}, \cite{abate-bracci-suwa-tovena}) and recently has found applications to the Sato hyperfunction
theory \cite{honda-izawa-suwa}.\\
In the present paper we define the relative \v{C}ech-Dolbeault homology of a complex manifold in terms of currents, and we prove in
Theorem \ref{thm:isom-relative-cohom-relative-hom} that
this description is equivalent to Suwa's one using complex differential forms. 
This different interpretation is used in Theorem \ref{thm:injectivity}
in order to prove a Wells-type result for relative cohomology, in particular we prove the following
\begin{theorem*}
Let $\pi:\tilde X\longrightarrow X$ be a proper, surjective, holomorphic map between two complex manifolds of the same dimension. Suppose that $X$ is connected and let $S$ and $\tilde S$ be closed complex submanifolds of $X$ and $\tilde X$ respectively, such that $\pi(\tilde S)\subset S$ and
$\pi(\tilde X\setminus\tilde S)\subset X\setminus S$.\\
Then,
$$
\pi^*:H^{p,q}_{\bar D}(X,X\setminus S)\longrightarrow 
H^{p,q}_{\bar D}(\tilde X,\tilde X\setminus \tilde S)
$$
is injective for any $p,q$.\\
\end{theorem*}
Differently from the classical Dolbeault theory, a dimensional inequality
in this setting cannot be expected even if $X$ and $\tilde X$ are compact,
indeed in general the relative Dolbeault cohomology groups of
a compact complex manifold are infinite-dimensional.
A similar result can be proved for the relative de Rham cohomology (cf. Theorem
\ref{thm:injectivity-derham}).

Notice that the hypothesis in the previous Theorem are satisfied by modifications.
In particular, if  $\tau:\tilde X
\to X$ is the blow-up of a complex manifold $X$ along a closed submanifold $Z$ then the previous assumptions are satisfied with $S=Z$ and
$\tilde S=E:=\pi^{-1}(Z)$ the exceptional divisor. Then, as a consequence the Dolbeault cohomology of $\tilde X$
can be expressed in terms of the Dolbeault cohomology of $X$
and the relative cohomology, more precisely (cf. Corollary \ref{cor:blow-up-formula})
\begin{corollary*}
Let $\tau:\tilde X
\to X$ be the blow-up of a complex manifold $X$ along a closed submanifold $Z$. Then,
there are isomorphisms 
$$
H^{p,q}_{\delbar}(\tilde X)\simeq
\tau^*H^{p,q}_{\delbar}(X)\oplus
\frac{H^{p,q}_{\bar D}(\tilde X,\tilde X\setminus E)}
{\tau^*H^{p,q}_{\bar D}(X,X\setminus Z)}\,,
$$
for any $p,\,q$, where $E$ is the exceptional divisor of the blow-up.\\
\end{corollary*}
This blow-up formula for the Dolbeault cohomology will find further applications in \cite{angella-suwa-tardini-tomassini}. For other formulas of this kind  we refer the reader to \cite{yang-yang}, \cite{rao-yang-yang-1}, \cite{rao-yang-yang-1}, \cite{stelzig-blowup}, \cite{meng}.

\medskip

\noindent {\em Acknowledgements.} The author would like to thank Daniele Angella, Tatsuo Suwa and Adriano Tomassini for many useful discussions and comments. The author would like to thank also Sheng Rao, Song Yang, Xiangdong Yang for their interest in the paper.
\smallskip

\section{Preliminaries and notations}

We start by fixing some notations and recalling some results about relative \v{C}ech-Dolbeault cohomology as presented in \cite{suwa}.\\
\textbf{\v{C}ech-Dolbeault cohomology\,.} 
Let $X$ be a complex manifold of complex dimension $n$. 
We denote with $A^{p,q}$ the space of smooth $(p,q)$-forms on $X$. 
Let $\mathcal{U}=\{U_{0},U_{1}\}$ be an open covering of $X$ and consider
\[
A^{p,q}(\mathcal{U}):=A^{p,q}(U_{0})\oplus A^{p,q}(U_{1})
\oplus A^{p,q-1}(U_{01})
\]
where by definition $U_{01}:=U_0\cap U_1$, with the differential operator $\bar D:A^{p,q}(\mathcal{U})\to A^{p,q+1}(\mathcal{U})$ defined on every element 
$(\xi_{0},\,\xi_{1},\,\xi_{01})\in A^{p,q}(\mathcal{U})$ by 
$$
\bar D\left(\xi_{0},\,\xi_{1},\,\xi_{01}\right)=\left(\delbar\xi_{0},\,\delbar\xi_{1},\,
\xi_{1}-\xi_{0}-\delbar\xi_{01}\right)\,.
$$
The \v{C}ech-Dolbeault cohomology associated to the covering $\mathcal{U}$
is defined by $H^{\bullet,\bullet}_{\bar D}(\mathcal{U})
=\text{Ker}\,\bar D/\text{Im}\,\bar D$
(cf. \cite{suwa} where this definition is given for an arbitrary open covering of $X$).
In \cite{suwa} it is proven that the 
morphism $A^{p,q}(X)\to A^{p,q}(\mathcal{U})$ given by $\xi\mapsto (\xi|_{U_{0}},\xi|_{U_{1}},0)$ induces an isomorphism
\[
H^{\bullet,\bullet}_{\delbar}(X)\overset\sim\to H^{\bullet,\bullet}_{\bar D}(\mathcal{U}),
\]
where $H^{p,q}_{\delbar}(X)$ denotes the Dolbeault cohomology of $X$.
In particular, the definition does not depend on the choice of the covering of $X$.
The inverse map is given by assigning to the class of $\xi=(\xi_{0},\xi_{1},\xi_{01})$ the class
of global $\delbar$-closed form $\rho_{0}\xi_{0}+\rho_{1}\xi_{1}-\delbar\rho_{0}\wedge\xi_{01}$, where $(\rho_{0},\rho_{1})$
is a partition of unity subordinate to the covering $\mathcal{U}$.\\
One can define naturally the cup product, integration on top-degree cohomology and the Kodaira-Serre duality and they turn out to be compatible with the above isomorphism (see \cite{suwa} for more details).

\textbf{Relative Dolbeault cohomology\,.} Let $S$ be a closed set in $X$. We let $U_{0}=X\setminus S$ and $U_{1}$ be an open neighborhood of $S$ in $X$ and consider the covering $\mathcal{U}=\{U_{0},U_{1}\}$ of $X$. We set,
for any $p,q$,
\[
A^{p,q}(\mathcal{U},U_{0}):=\{\,\xi\in A^{p,q}(\mathcal{U})\mid\xi_{0}=0\,\}=A^{p,q}(U_{1})\oplus A^{p,q-1}(U_{01}).
\]
Then $\left(A^{p,\bullet}(\mathcal{U},U_{0}), \,\bar D\right)$ is a subcomplex of 
$\left(A^{p,\bullet}(\mathcal{U}),\,\bar D\right)$. 
Let $H^{p,q}_{\bar D}(\mathcal{U},U_{0})$ be the
cohomology of $\left(A^{p,\bullet}(\mathcal{U},U_{0}),\,\bar D\right)$. From the exact sequence 
\[
0\to A^{p,\bullet}(\mathcal{U},U_{0})\to A^{p,\bullet}(\mathcal{U})
\to A^{p,\bullet}(U_{0})\to 0
\]
where the first map is the inclusion and the second map is the projection on the first element, we obtain the long exact sequence in cohomology
$$
\cdots\to H^{p,q-1}_{\delbar}(U_{0})\overset\delta\to H^{p,q}_{\bar D}(\mathcal{U},U_{0})\overset{j^{*}}\to H^{p,q}_{\bar D}(\mathcal{U})\overset{i^{*}}\to H^{p,q}_{\delbar}(U_{0})\to\cdots.
$$
Therefore, $H^{\bullet,\bullet}_{\bar D}(\mathcal{U},U_{0})$ is determined uniquely modulo
canonical isomorphism. Thus we denote it with $H^{\bullet,\bullet}_{\bar D}(X,X\setminus S)$
and we call it the \emph{relative Dolbeault cohomology} of $X$.\\
Together with integration theory the relative Dolbeault cohomology has been used
to study the localization of characteristic classes (cf. \cite{suwa}, 
\cite{abate-bracci-suwa-tovena}) and has found more recent
applications to hyperfunction theory (cf. \cite{honda-izawa-suwa}).\\
However notice that the explicit computation of these relative Dolbeault cohomology groups can be very difficult even in some easy situations.\\

If $X$ and $\tilde X$ are complex manifolds, $S$ and $\tilde S$ are closed subsets
in $X$ and $\tilde X$ respectively and $f:\tilde X\to X$ is a holomorphic map such that $f(\tilde S)\subset S$ and
$f(\tilde X\setminus \tilde S)\subset f(X\setminus S)$, 
then $f$ induces a natural map in relative cohomology.
Indeed, setting $U_0:= X\setminus S$, $\tilde U_0:= \tilde X\setminus \tilde S$
and let $U_1$, $\tilde U_1$ be open neighborhoods of $S$ and $\tilde S$
in $X$ and $\tilde X$ respectively, chosen in such a way that
$f(\tilde U_1)\subset U_1$. Take the open coverings
$\mathcal{U}:=\left\lbrace U_0,U_1\right\rbrace$
and $\mathcal{\tilde U}:=\left\lbrace \tilde U_0,\tilde U_1\right\rbrace$
of $X$ and $\tilde X$ respectively, then we have a homomorphism
$$
f^*:A^{\bullet,\bullet}(\mathcal{U},U_0)\to 
A^{\bullet,\bullet}(\mathcal{\tilde U},\tilde U_0)
$$
defined on every element $(\xi_1,\xi_{01})\in A^{\bullet,\bullet}(\mathcal{U},U_0)$ as
$$
f^*(\xi_1,\xi_{01}):=(f^*\xi_1,f^*\xi_{01})
$$
which induces a homomorphism in relative cohomology
$$
f^*:H^{\bullet,\bullet}_{\bar D}(X,X\setminus S)\to
H^{\bullet,\bullet}_{\bar D}(\tilde X,\tilde X\setminus \tilde S)\,.
$$

\section{Relative \v{C}ech-Dolbeault homology}

In this Section we describe a \v{C}ech interpretation  of the Dolbeault homology and we give a definition for its relative counterpart. Finally in Theorem
\ref{thm:isom-relative-cohom-relative-hom} we prove that the relative Dolbeault cohomology can be computed equivalently using forms and currents.\\

\textbf{\v{C}ech-Dolbeault homology.}
Before defining the relative Dolbeault homology we discuss a \v{C}ech interpretation of the Dolbeault homology.
Let $X$ be a complex manifold of complex dimension $n$ and denote with $\mathcal{K}^{p,q}(X)=\mathcal{K}_{n-p,n-q}(X)$ the space of currents of bidegree $(p,q)$, or equivalently of bidimension $(n-p,n-q)$, on $X$,
namely the topological dual of the space of $(n-p,n-q)$-forms with compact support in $X$.\\
The differential operator $\delbar:\mathcal{K}^{p,q}(X)\to
\mathcal{K}^{p,q+1}(X)$ is defined as usual, for any
$T\in \mathcal{K}^{p,q}(X)$, $\varphi\in A^{n-p,n-q-1}(X)$ with compact support, as
$$
\left\langle\delbar T,\varphi\right\rangle:= (-1)^{p+q+1}
\left\langle T,\delbar \varphi\right\rangle\,,
$$
where $\left\langle -\,,\,- \right\rangle$ stands for the duality pairing.
Then for any $p$, we denote with $H^{p,\bullet}_{\delbar_\mathcal{K}}(X)$ the cohomology
of the complex $(\mathcal{K}^{p,\bullet}(X),\delbar)$ which is called
the \emph{Dolbeault homology of $X$}.\\

Now let $\mathcal{U}=\left\lbrace U_\alpha\right\rbrace_{\alpha\in I}$ be an open covering of $X$ where $I$ is an ordered set and let
$I^{(r)}:=\left\lbrace (\alpha_0,\cdots,\alpha_r)\,\mid\, \alpha_0<
\cdots<\alpha_r,\,\alpha_{\nu}\in I\right\rbrace$. 
We set, for any $r,p,q$,
$$
B^r(\mathcal{U},\mathcal{K}^{p,q})\,:=\,\Pi_{(\alpha_0,\cdots,\alpha_r)
\in I^{(r)}}
\mathcal{K}^{p,q}(U_{\alpha_0\cdots\alpha_r})
$$
where, by definition, $U_{\alpha_0\cdots\alpha_r}:=U_{\alpha_0}\cap\cdots\cap
U_{\alpha_r}$, and we define the boundary operator as
$$
\delta:B^r(\mathcal{U},\mathcal{K}^{p,q})\to
B^{r+1}(\mathcal{U},\mathcal{K}^{p,q})
$$
$$
(\delta T)_{\alpha_0\cdots\alpha_{r+1}}:=\sum_{\nu=0}^{r+1}(-1)^\nu T_{\alpha_0\cdots\hat{\alpha}_\nu\cdots\alpha_{r+1}}
$$
where all the currents $T_{\alpha_0\cdots\hat{\alpha}_\nu\cdots\alpha_{r+1}}$
have to be restricted to $U_{\alpha_0\cdots\alpha_{r+1}}$.
Moreover, setting
$\delbar:B^r(\mathcal{U},\mathcal{K}^{p,q})\to B^r(\mathcal{U},\mathcal{K}^{p,q+1})$ the extension of the operator $\delbar$ to every components one gets that $B^\bullet(\mathcal{U},\mathcal{K}^{p,\bullet})$ endowed with the operators $\delta$ and $\delbar$ is a double complex. The associated total complex
will be denoted with
$(\mathcal{K}^{p,\bullet}(\mathcal{U}),\bar D_{\mathcal{K}})$, namely
$$
\mathcal{K}^{p,q}(\mathcal{U}):=\oplus_{s+r=q}B^r(\mathcal{U},\mathcal{K}^{p,s})
$$
and
$$
\bar D_{\mathcal{K}}(T_{\alpha_0\cdots\alpha_{r}})=
\sum_{\nu=0}^r(-1)^\nu
T_{\alpha_0\cdots\hat{\alpha}_\nu\cdots\alpha_{r}}+
(-1)^r\,\delbar T_{\alpha_0\cdots\alpha_{r}}\,.
$$
In particular, for $r=1$ we have
$$
\bar D_{\mathcal{K}}(T_{\alpha_0\,\alpha_{1}}):=
T_{\alpha_{1}}-T_{\alpha_{0}}-
\delbar \,T_{\alpha_0\alpha_{1}}\,.
$$
\begin{definition}
The cohomology of the complex $(\mathcal{K}^{p,\bullet}(\mathcal{U}),\bar D_{\mathcal{K}})$ will be denoted with 
$H_{\bar D_{\mathcal{K}}}^{\bullet,\bullet}(\mathcal{U})$ and will be called
\emph{\v{C}ech Dolbeault homology} associated to the covering $\mathcal{U}$.
\end{definition}
This definition does not depend on the open covering, indeed one has the following
\begin{theorem}\label{isom-chec-dolbeault}
The restriction map $\mathcal{K}^{p,q}(X)\to B^0(\mathcal{U},\mathcal{K}^{p,q})$ induces a natural isomorphism
$$
H^{p,q}_{\delbar_\mathcal{K}}\to H_{\bar D_{\mathcal{K}}}^{p,q}(\mathcal{U}).
$$
\end{theorem}
\begin{proof}
The argument works exactly as in \cite{suwa}, since the complex of $(p,q)$-currents on $\mathcal{U}$ is acyclic, being a $\mathcal{C}^{\infty}(X)$-module.
For completeness we recall here the proof. We consider the first spectral sequence associated to the double complex
$B^\bullet(\mathcal{U},\mathcal{K}^{p,\bullet})$. In particular, at the second page one has
$$
'E^{q,r}_2:=H^q_{\delbar}H^r_{\delta}\left(
B^\bullet(\mathcal{U},\mathcal{K}^{p,\bullet})\right)
$$
which converges to $H^{p,q+r}_{\bar D_{\mathcal{K}}}(\mathcal{U})$.
By the acyclic property, for $r>0$, $H^r_{\delta}\left(
B^\bullet(\mathcal{U},\mathcal{K}^{p,\bullet})\right)=0$ and for $r=0$,
$H^0_{\delta}\left(
B^\bullet(\mathcal{U},\mathcal{K}^{p,\bullet})\right)=
\mathcal{K}^{p,\bullet}(X)$. Therefore,
$H_{\bar D_{\mathcal{K}}}^{p,q}(\mathcal{U})\simeq\,
'E^{q,0}_2\simeq H^{p,q}_{\delbar_\mathcal{K}}(X)$.
\end{proof}
\begin{rem}
Since the Dolbeault cohomology can be computed using currents, it follows by the previous Theorem that there is also an isomorphism with the Dolbeault cohomology of $X$, namely
$$
H^{\bullet,\bullet}_{\delbar}(X)\,\simeq \,
H^{\bullet,\bullet}_{\delbar_\mathcal{K}}(X)\,\simeq\,
H_{\bar D_{\mathcal{K}}}^{\bullet,\bullet}(\mathcal{U})\,.
$$
Since to every $(p,q)$-form $\varphi$ on $X$ we can associate a $(p,q)$-current 
$i(\varphi):=\int_X \varphi\wedge\cdot$, then we have a natural injective map
$$
i:A^{p,q}(\mathcal{U})\to \mathcal{K}^{p,q}(\mathcal{U})
$$
where by definition (see \cite{suwa} for more details)
$$
A^{p,q}(\mathcal{U}):=\oplus_{s+r=q}\left(
\Pi_{(\alpha_0,\cdots,\alpha_r)
\in I^{(r)}}
A^{p,s}(U_{\alpha_0\cdots\alpha_r})\right)\,.
$$
\end{rem}

\textbf{Relative \v{C}ech-Dolbeault homology.}
Now we define the relative \v{C}ech Dolbeault homology. Let $S$ be a closed set in $X$. We let $U_{0}=X\setminus S$ and $U_{1}$ be an open neighborhood of $S$ in $X$ and consider the open covering $\mathcal{U}=\{U_{0},U_{1}\}$ of $X$. 
In particular, in this situation we have
$$
\mathcal{K}^{p,q}(\mathcal{U})=B^0(\mathcal{U},\mathcal{K}^{p,q})
\oplus B^1(\mathcal{U},\mathcal{K}^{p,q-1})
$$
and an element of $\mathcal{K}^{p,q}(\mathcal{U})$ can be written as a triple
$(T_0,T_1,T_{01})$ where $T_0$ is a $(p,q)$-current on $U_0$,
$T_1$ is a $(p,q)$-current on $U_1$ and 
$T_{01}$ is a $(p,q-1)$-current on $U_{01}$.\\
We set
\[
\mathcal{K}^{p,q}(\mathcal{U},U_{0}):=
\left\lbrace T\in\mathcal{K}^{p,q}(\mathcal{U})\,\mid\, T_0=0\right\rbrace
=\mathcal{K}^{p,q}(U_{1})\oplus\mathcal{K}^{p,q-1}(U_{01}).
\]
Then $\mathcal{K}^{p,\bullet}(\mathcal{U},U_{0})$ is a subcomplex of $\mathcal{K}^{p,\bullet}(\mathcal{U})$. 
Therefore we have the following short exact sequence
\[
0\to \mathcal{K}^{p,\bullet}(\mathcal{U},U_{0})\to 
\mathcal{K}^{p,\bullet}(\mathcal{U})\to \mathcal{K}^{p,\bullet}(U_{0})\to 0
\]
where the first map is the inclusion and the second map is the projection on the first factor and clearly by definition
$$
\bar D_{\mathcal{K}}:\mathcal{K}^{p,q}(\mathcal{U},U_{0})\to \mathcal{K}^{p,q+1}(\mathcal{U},U_{0})\,.
$$
We denote with $H^{p,q}_{\bar D_{\mathcal{K}}}(\mathcal{U},U_{0})$ the associated cohomology.
We get the following long exact sequence in homology
$$
\cdots\to H^{p,q-1}_{\delbar_{\mathcal{K}}}(U_{0})\overset\delta
\to H^{p,q}_{\bar D_{\mathcal{K}}}(\mathcal{U},U_{0})\overset{j^{*}}\to
 H^{p,q}_{\bar D_{\mathcal{K}}}(\mathcal{U})\overset{i^{*}}\to
  H^{p,q}_{\delbar_{\mathcal{K}}}(U_{0})\to\cdots.
$$
By Theorem \ref{isom-chec-dolbeault} we see that $H^{p,q}_{\delbar_{\mathcal{K}}}(\mathcal{U},U_{0})$ is determined uniquely modulo
canonical isomorphisms. 
\begin{definition}
In the above situation we set
$H^{p,q}_{\bar D_{\mathcal{K}}}(X,X\setminus S):=H^{p,q}_{D_{\mathcal{K}}}(\mathcal{U},X\setminus S)$ and we call it the
\emph{relative \v{C}ech-Dolbeault homology} of $X$ with respect to
$X\setminus S$.
\end{definition}
Notice that we have an injective map
$$
i:A^{p,q}(\mathcal{U},U_0)\to \mathcal{K}^{p,q}(\mathcal{U},U_0)
$$
which induces naturally a map
$$
H^{p,q}_{\bar D}(X,X\setminus S)\to
H^{p,q}_{\bar D_{\mathcal{K}}}(X,X\setminus S)
$$
from the relative Dolbeault cohomology to the relative Dolbeault homology.

\begin{theorem}\label{thm:isom-relative-cohom-relative-hom}
The map
$$
H^{p,q}_{\bar D}(X,X\setminus S)\to
H^{p,q}_{\bar D_{\mathcal{K}}}(X,X\setminus S)
$$
is an isomorphism.
\end{theorem}
\begin{proof}
We have the following commutative diagram between forms and currents
$$
\xymatrix{
0\ar[r] &A^{p,q}(\mathcal{U},U_{0})\ar[r] \ar[d]&
A^{p,q}(\mathcal{U})\ar[r] \ar[d]& 
A^{p,q}(U_{0})\ar[r]\ar[d] & 0\\
0\ar[r] &\mathcal{K}^{p,q}(\mathcal{U},U_{0})\ar[r] &
\mathcal{K}^{p,q}(\mathcal{U})\ar[r] & 
\mathcal{K}^{p,q}(U_{0})\ar[r] & 0
}
$$
which induces the commutative diagram with exact rows
$$
{\resizebox{\textwidth}{!}{
\xymatrix{
\cdots \ar[r] & H^{p,q-1}_{\delbar}(X) \ar[r] \ar[d]^{\simeq} & 
H^{p,q-1}_{\delbar}(X\setminus S) \ar[r] \ar[d]^{\simeq} & 
H^{p,q}_{\bar D}(X,X\setminus S) \ar[r] \ar[d] & 
H^{p,q}_{\delbar}(X) \ar[r] \ar[d]^{\simeq} &
H^{p,q}_{\delbar}(X\setminus S) \ar[r] \ar[d]^{\simeq} & \cdots \\
\cdots \ar[r] & H^{p,q-1}_{\delbar_\mathcal{K}}(X) \ar[r] & 
H^{p,q-1}_{\delbar_\mathcal{K}}(X\setminus S) \ar[r] & 
H^{p,q}_{\bar D_{\mathcal{K}}}(X,X\setminus S) \ar[r] & 
H^{p,q}_{\delbar_\mathcal{K}}(X) \ar[r] &
H^{p,q}_{\delbar_\mathcal{K}}(X\setminus S)\ar[r] & \cdots 
}
}}
$$
Since the Dolbeault cohomology of $X$ and $X\setminus S$ can be computed equivalently using differential forms or currents we have that the first two and the last two vertical maps are isomorphisms. By the five lemma we obtain the same conclusion for the central vertical morphism.
\end{proof}

Now let $\pi:\tilde X\longrightarrow X$ be a proper holomorphic map between two complex manifolds and let $S$ be a closed subset of $X$ and $\tilde S$ a
closed subset of $\tilde X$ and take $U_0=X\setminus S$ and $U_1$ a neighborhood of $S$ in $X$, similarly $\tilde U_0=
\tilde X\setminus \tilde S$ and $\tilde U_1$ a neighborhood of $\tilde 
S$ in $\tilde X$. Suppose that $\pi (\tilde S)\subset S$,
$\pi (\tilde U_0)\subset U_0$ and
$\pi (\tilde U_1)\subset U_1$. Then $\mathcal{U}:=
\left\lbrace U_0,U_1\right\rbrace$ and $\mathcal{\tilde U}:=
\left\lbrace \tilde U_0,\tilde U_1\right\rbrace$ are compatible open coverings of $X$
and $\tilde X$ respectively. We define the \emph{push-forward} 
$\pi_*:\mathcal{K}^{p,q}(\mathcal{\tilde U},\tilde U_0)\to
\mathcal{K}^{p,q}(\mathcal{U}, U_0)$
namely
$$
\pi_*:\mathcal{K}^{p,q}(\tilde U_1)\oplus
\mathcal{K}^{p,q-1}(\tilde U_{01})\to
\mathcal{K}^{p,q}(U_1)\oplus
\mathcal{K}^{p,q-1}(U_{01})
$$
as
$$
(T_1,T_{01})\mapsto (\pi_*T_1,\pi_*T_{01})\,.
$$
In particular, the push-forward commutes with the differential
$\bar D_{\mathcal{K}}$, indeed
$$
\pi_*\bar D_{\mathcal{K}}(T_1,T_{01})=
(\pi_*\delbar T_1,\pi_*T_1-\pi_*\delbar T_{01})=
(\delbar\pi_* T_1,\pi_*T_1-\delbar \pi_*T_{01})=
\bar D_{\mathcal{K}}\pi_*(T_1,T_{01})\,,
$$
where in the second equality we use that $\pi_*$ commutes with 
the operator $\delbar$.
Hence, the push-forward induces a map in relative Dolbeault homology
$$
\pi_*:H^{\bullet,\bullet}_{\bar D_{\mathcal{K}}}
(\tilde X,\tilde X\setminus \tilde S)\to
H^{\bullet,\bullet}_{\bar D_{\mathcal{K}}}(X, X\setminus  S)\,.
$$
We will use this map in the next Section in order to prove that under suitable assumptions the pull-back map in relative cohomology is injective.

\begin{rem}
Similar considerations can be done for the relative \v{C}ech-de Rham cohomology (cf. \cite{suwa-book} for its definition) using $k$-currents with compact support and the exterior derivative $d$.
\end{rem}

\section{Comparisons via proper holomorphic surjective maps}

In this Section we study a Wells-type result for relative Dolbeault cohomology,
in particular we prove the following
\begin{theorem}\label{thm:injectivity}
Let $\pi:\tilde X\longrightarrow X$ be a proper, surjective, holomorphic map between two complex manifolds of the same dimension. Suppose that $X$ is connected and let $S$ and $\tilde S$ be closed complex submanifolds of $X$ and $\tilde X$ respectively, such that $\pi(\tilde S)\subset S$ and
$\pi(\tilde X\setminus\tilde S)\subset X\setminus S$.\\
Then,
$$
\pi^*:H^{p,q}_{\bar D}(X,X\setminus S)\longrightarrow 
H^{p,q}_{\bar D}(\tilde X,\tilde X\setminus \tilde S)
$$
is injective for any $p,q$.\\
\end{theorem}

\begin{proof}
We denote with $U_0=X\setminus S$ and $\tilde U_0=
\tilde X\setminus \tilde S$. Let
$U_1$ be a neighborhood of $S$ in $X$ and $\tilde U_1$ be a neighborhood of $\tilde 
S$ in $\tilde X$ such that $\pi(\tilde U_1)= U_1$ and
$\pi(\tilde U_{01})=U_{01}$. Hence we
consider the open coverings $\mathcal{U}:=
\left\lbrace U_0,U_1\right\rbrace$ of $X$ and $\mathcal{\tilde U}:=
\left\lbrace \tilde U_0,\tilde U_1\right\rbrace$ of $\tilde X$.\\
We take the following diagram, for any $p,q$
$$
\xymatrixcolsep{5pc}\xymatrix{
A^{p,q}(\mathcal{\tilde U},\tilde U_0) \ar[r]^{\tilde i} &
\mathcal{K}^{p,q}(\mathcal{\tilde U},\tilde U_0)\ar[d]^{\pi_*}\\
A^{p,q}(\mathcal{U}, U_0) \ar[u]^{\pi^*}\ar[r]^{i}   &
\mathcal{K}^{p,q}(\mathcal{U}, U_0)\,,
}
$$
or equivalently, by definition
$$
\xymatrixcolsep{5pc}\xymatrix{
A^{p,q}(\tilde U_{1})\oplus A^{p,q-1}(\tilde U_{01})
 \ar[r]^{\tilde i} &
\mathcal{K}^{p,q}(\tilde U_{1})\oplus \mathcal{K}^{p,q-1}(\tilde U_{01})\ar[d]^{\pi_*}\\
A^{p,q}(U_{1})\oplus A^{p,q-1}(U_{01}) \ar[u]^{\pi^*}\ar[r]^{i}   &
\mathcal{K}^{p,q}(U_{1})\oplus \mathcal{K}^{p,q-1}(U_{01})\,,
}
$$
where $\tilde i$ and $i$ denote the natural injections of forms into currents.
By \cite[Lemma 2.1]{wells} we have that the diagram
$$
\xymatrixcolsep{5pc}\xymatrix{
A^{p,q}(\tilde U_{1})
 \ar[r]^{\tilde i} &
\mathcal{K}^{p,q}(\tilde U_{1})\ar[d]^{\pi_*}\\
A^{p,q}(U_{1}) \ar[u]^{\pi^*}\ar[r]^{i}   &
\mathcal{K}^{p,q}(U_{1})\,.
}
$$
commutes up to a constant, more precisely we have
$\mu\, i=\pi_*\tilde i\pi^*$, where $\mu$ is the degree of $\pi$. Recall that the degree of $\pi$ is defined as $\pi_*(1)$, where $1$ is thought as a current on $\tilde X$, in particular it is a $d$-closed function on $X$. By the connectedness of $X$ we have that the degree of $\pi$ is constant on $X$.
Similarly, we also have the commutativity up to $\mu$, of
$$
\xymatrixcolsep{5pc}\xymatrix{
A^{p,q-1}(\tilde U_{01})
 \ar[r]^{\tilde i} &
\mathcal{K}^{p,q-1}(\tilde U_{01})\ar[d]^{\pi_*}\\
 A^{p,q-1}(U_{01}) \ar[u]^{\pi^*}\ar[r]^{i}   &
 \mathcal{K}^{p,q-1}(U_{01})\,.
}
$$
Therefore, one can pass to (co)homology considering the following
$$
\xymatrixcolsep{5pc}\xymatrix{
H^{p,q}_{\bar D}(\tilde X,\tilde X\setminus\tilde S)
 \ar[r]^{\tilde i_*} &
H^{p,q}_{\bar D_{\mathcal{K}}}
(\tilde X,\tilde X\setminus\tilde S)\ar[d]^{\pi_*}\\
 H^{p,q}_{\bar D}(X, X\setminus S)\ar[u]^{\pi^*}\ar[r]^{i_*}   &
 H^{p,q}_{\bar D_{\mathcal{K}}}(X, X\setminus S)\,.
}
$$
We have shown in Theorem \ref{thm:isom-relative-cohom-relative-hom}
that the maps $\tilde i_*$ and $i_*$ in this last diagram are isomorphisms.\\
Using this fact we prove that $\pi^*$ is injective, indeed
let $a\in H^{p,q}(X, X\setminus S)$ and
suppose that $\pi^*a=0$, then $\mu\, i_*a=\pi_*\tilde i_*\pi^*a=0$. Then $i_*a=0$ and by injectivity we can conclude that $a=0$, proving the assertion.
\end{proof}

As already noticed, one can define the relative \v{C}ech-de Rham homology and with similar techniques one can prove an injectivity result also for the relative \v{C}ech-de Rham cohomology (cf. \cite[Theorem 3.1]{wells}). For completeness we write down the Theorem explicitly without the proof since it is the same for the Dolbeault case.
\begin{theorem}\label{thm:injectivity-derham}
Let $\pi:\tilde X\longrightarrow X$ be a proper, surjective, holomorphic map between two complex manifolds of the same dimension. Suppose that $X$ is connected and let $S$ and $\tilde S$ be closed complex submanifolds of $X$ and $\tilde X$ respectively, such that $\pi(\tilde S)\subset S$ and
$\pi(\tilde X\setminus\tilde S)\subset X\setminus S$.\\
Then,
$$
\pi^*:H^{k}_{dR}(X,X\setminus S)\longrightarrow 
H^{k}_{dR}(\tilde X,\tilde X\setminus \tilde S)
$$
is injective for any $k$.\\
\end{theorem}

\begin{rem}
Notice that even if $\tilde X$ and $X$ are compact the relative cohomology groups can be infinite-dimensional, therefore in general we don't have a comparison between the dimensions of these cohomology groups.
\end{rem}

\section{Application to blow-ups}\label{rmk:blow-up}

In these last years an increased interest for blow-up formulas for cohomology groups of complex manifolds arised with the purpose of studying the behavior of the $\partial\overline\partial$-lemma under modifications.
Let $\tau:\tilde X\to X$ be the blow-up of a compact complex manifold
$X$ along a closed complex submanifold $Z$. Then, by a classical result if $X$ is K\"ahler the same holds for $\tilde X$ and one can express the de Rham cohomology of $\tilde X$ in terms of the de Rham cohomology groups of $X$ and $Z$ (cf. \cite{voisin}). In fact, the K\"ahler hypothesis can be dropped and several formulas for the de Rham, Dolbeault and Bott-Chern cohomologies can be found in 
\cite{stelzig-blowup}, \cite{yang-yang}, \cite{rao-yang-yang-1}, \cite{rao-yang-yang-1}, \cite{meng}, \cite{angella-suwa-tardini-tomassini}. Here we discuss a blow-up formula for the Dolbeault cohomology in terms of the relative cohomology groups which follows directly by the previous results and which finds further applications in \cite{angella-suwa-tardini-tomassini}.

Let $X$ be a compact complex manifold and let $Z$ be a closed complex submanifold. Then, the blow-up $\tau:\tilde X
\to X$ of $X$ along $Z$ satisfies the previous assumptions with $S=Z$ and
$\tilde S=E:=\pi^{-1}(Z)$ the exceptional divisor.
Then, by Theorem \ref{thm:injectivity} the induced maps in relative cohomology
$$
\tau^*:H^{p,q}_{\bar D}(X,X\setminus Z)\longrightarrow 
H^{p,q}_{\bar D}(\tilde X,\tilde X\setminus E)
$$
are injective for any $p,\,q$.\\
Therefore, one has the following commutative diagram with exact rows
(cf. \cite{angella-suwa-tardini-tomassini})
$$
{\resizebox{\textwidth}{!}{
\xymatrix{
\cdots\ar[r] & H^{p,q-1}_{\delbar}(X\setminus Z)\ar[r]\ar[d]^-{\tau^{*}}& 
H^{p,q}_{\bar D}(X,X\setminus Z)\ar[d]^-{\tau^{*}}\ar[r]&
H^{p,q}_{\delbar}(X)\ar[r]\ar[d]^-{\tau^{*}}&
H^{p,q}_{\delbar}(X\setminus Z)\ar[r]\ar[d]^-{\tau^{*}}&
H^{p,q+1}_{\bar D}(X,X\setminus Z)\ar[r]\ar[d]^-{\tau^{*}}&
\cdots\\
 \cdots\ar[r] & H^{p,q-1}_{\delbar}(\tilde X\setminus E)\ar[r] &
  H^{p,q}_{\bar D}(\tilde X,\tilde X\setminus E)\ar[r]&
  H^{p,q}_{\delbar}(\tilde X)\ar[r] &
  H^{p,q}_{\delbar}(\tilde X\setminus E)\ar[r]&
  H^{p,q+1}_{\bar D}(\tilde X,\tilde X\setminus E)\ar[r]&
  \cdots
 }
 }}
$$
where the maps
$$
\tau^*:H^{\bullet,\bullet}_{\delbar}(X\setminus Z)
\to H^{\bullet,\bullet}_{\delbar}(\tilde X\setminus E)
$$
are isomorphisms since $X$ and $\tilde X$ are biolomorphic outside the exceptional divisor. Furthermore, the maps
$$
\tau^*:H^{\bullet,\bullet}_{\delbar}(X)
\to H^{\bullet,\bullet}_{\delbar}(\tilde X)
$$
are injective by Wells' Theorem \cite[Theorem 3.1]{wells} and finally the maps
$$
\tau^*:H^{\bullet,\bullet}_{\bar D}(X,X\setminus Z)
\to H^{\bullet,\bullet}_{\bar D}(\tilde X,\tilde X\setminus E)
$$
are injective by Theorem \ref{thm:injectivity}. 
As a consequence, for instance by \cite[Lemme II.6]{blanchard}, one obtains the following isomorphisms
\begin{corollary}\label{cor:blow-up-formula}
Let $\tau:\tilde X\to X$ be the blow-up of a compact complex manifold
$X$ along a closed complex submanifold $Z$ and let $E$ be the exceptional divisor. Then, there are isomorphisms
$$
H^{p,q}_{\delbar}(\tilde X)\simeq
\tau^*H^{p,q}_{\delbar}(X)\oplus
\frac{H^{p,q}_{\bar D}(\tilde X,\tilde X\setminus E)}
{\tau^*H^{p,q}_{\bar D}(X,X\setminus Z)}\,,
$$
for any $p,\,q$.
\end{corollary}
Notice that in \cite{rao-yang-yang-1} and \cite{rao-yang-yang-2} the authors prove
a Dolbeault blow-up formula using a sheaf-theoretic approach. However their definition of relative Dolbeault cohomology is different from Suwa's one.

\begin{rem}
In fact the previous formula is a bit more general.
Let $\pi:\tilde X\longrightarrow X$ be a proper, surjective, holomorphic map between two complex manifolds of the same dimension. Suppose that $X$ is connected and let $S$ and $\tilde S$ be closed complex submanifolds of $X$ and $\tilde X$ respectively, such that $\pi(\tilde S)\subset S$ and
$\pi(\tilde X\setminus\tilde S)\subset X\setminus S$.
Fix $p$ and $q$ and consider the following diagram with exact rows
$$
{\resizebox{\textwidth}{!}{
\xymatrix{
\cdots\ar[r] & H^{p,q-1}_{\delbar}(X\setminus S)\ar[r]\ar[d]^-{\pi^{*}}& 
H^{p,q}_{\bar D}(X,X\setminus S)\ar[d]^-{\pi^{*}}\ar[r]&
H^{p,q}_{\delbar}(X)\ar[r]\ar[d]^-{\pi^{*}}&
H^{p,q}_{\delbar}(X\setminus S)\ar[r]\ar[d]^-{\pi^{*}}&
H^{p,q+1}_{\bar D}(X,X\setminus S)\ar[r]\ar[d]^-{\pi^{*}}&
\cdots\\
 \cdots\ar[r] & H^{p,q-1}_{\delbar}(\tilde X\setminus \tilde S)\ar[r] &
  H^{p,q}_{\bar D}(\tilde X,\tilde X\setminus \tilde S)\ar[r]&
  H^{p,q}_{\delbar}(\tilde X)\ar[r] &
  H^{p,q}_{\delbar}(\tilde X\setminus \tilde S)\ar[r]&
  H^{p,q+1}_{\bar D}(\tilde X,\tilde X\setminus \tilde S)\ar[r]&
  \cdots
 }
 }}
$$
Then, as above the maps
$$
\tau^*:H^{\bullet,\bullet}_{\delbar}(X)
\to H^{\bullet,\bullet}_{\delbar}(\tilde X)
\quad\text{and}\quad
\tau^*:H^{\bullet,\bullet}_{\bar D}(X,X\setminus S)
\to H^{\bullet,\bullet}_{\bar D}(\tilde X,\tilde X\setminus \tilde S)
$$
are injective respectively by \cite[Theorem 3.1]{wells} and 
Theorem \ref{thm:injectivity}. 
If
$$
\pi^*:H^{p,q-1}_{\delbar}(X\setminus S)\to
H^{p,q-1}_{\delbar}(\tilde X\setminus \tilde S)
$$
is surjective and
$$
\pi^*:H^{p,q}_{\delbar}(X\setminus S)\to
H^{p,q}_{\delbar}(\tilde X\setminus \tilde S)
$$
is an isomorphism, then there is an isomorphism
$$
H^{p,q}_{\delbar}(\tilde X)\simeq
\pi^*H^{p,q}_{\delbar}(X)\oplus
\frac{H^{p,q}_{\bar D}(\tilde X,\tilde X\setminus \tilde S)}
{\pi^*H^{p,q}_{\bar D}(X,X\setminus S)}\,.
$$
Namely the relative cohomology groups measure the gap between
$H^{p,q}_{\delbar}(\tilde X)$ and $H^{p,q}_{\delbar}(X)$.
\end{rem}

\end{document}